\documentclass[a4paper,oneside, 11 pt, reqno]{amsart}
\usepackage{enumerate}
\usepackage{amsmath,amsthm,amscd,amsfonts,amssymb,eufrak,mathrsfs}
\makeatletter
\@namedef{subjclassname@2010}{%
  \textup{2010} Mathematics Subject Classification}
\makeatother

\newcommand{\ncom}{\newcommand}
\ncom{\beqn}{\begin{eqnarray*}}
\ncom{\eeqn}{\end{eqnarray*}}
\ncom{\beq}{\begin{eqnarray}}
\ncom{\eeq}{\end{eqnarray}}
\ncom{\cal}{\mathcal}
\ncom{\eop}{\hfill{{\rule{2.5mm}{2.5mm}}}}
\ncom{\eoe}{\hfill{{\rule{1.5mm}{1.5mm}}}}
\ncom{\eof}{\hfill{{\rule{1.5mm}{1.5mm}}}}
\ncom{\hone}{\mbox{\hspace{1em}}}
\ncom{\htwo}{\mbox{\hspace{2em}}}
\ncom{\hthree}{\mbox{\hspace{3em}}}
\ncom{\hfour}{\mbox{\hspace{4em}}}
\ncom{\hsev}{\mbox{\hspace{7em}}}
\ncom{\vone}{\vskip 2ex}
\ncom{\vtwo}{\vskip 4ex}
\ncom{\vonee}{\vskip 1.5ex}
\ncom{\vthree}{\vskip 6ex}
\ncom{\vfour}{\vspace*{8ex}}
\ncom{\norm}{\|\;\;\|}
\ncom{\integ}[4]{\int_{#1}^{#2}\,{#3}\,d{#4}}
\ncom{\inp}[2]{\langle{#1},\,{#2} \rangle}
\ncom{\Inp}[2]{\Langle{#1},\,{#2} \Langle}
\ncom{\vspan}[1]{{{\rm\,span}\#1 \}}}
\ncom{\dm}[1]{\displaystyle {#1}}

\newtheorem{theorem}{\bf Theorem}[section]
\newtheorem{proposition}[theorem]{\bf Proposition}
\newtheorem{lemma}[theorem]{\bf Lemma}

\newtheoremstyle
    {remarkstyle}
    {}
    {11pt}
    {}
    {}
    {\bfseries}
    {:}
    {     }
    {\thmname{#1} \thmnumber{#2} }

\theoremstyle{remarkstyle}

\newtheorem{remark}[theorem]{\bf Remark}
\newtheorem{definition}[theorem]{\bf Definition}
\newtheorem{example}[theorem]{\bf Example}

\begin{document}

\baselineskip=17pt

\title[Weighted Translation Semigroups: Multivariable Case]{Weighted Translation Semigroups: Multivariable Case}

\author[G. M. Phatak]{Geetanjali M. Phatak}
\address{Department of Mathematics, S. P. College\\
Pune- 411030, India}
\email{gmphatak19@gmail.com}

\author[V. M. Sholapurkar]{V. M. Sholapurkar}
\address{Head, Department of Mathematics, S. P. College\\
Pune- 411030, India}
\email{vmshola@gmail.com}

\date{}

\begin{abstract}
M. Embry and A. Lambert initiated the study of a weighted translation semigroup $\{S_t\}$ in ${\cal B}(L^2({\mathbb R_+})),$ with a view to explore a continuous analogue of a weighted shift operator. We continued the work, characterized some special types of semigroups and developed an analytic model for the left invertible weighted translation semigroup. The present paper deals with the generalization of the weighted translation semigroup in multi-variable set up. We develop the toral analogue of the analytic model and also describe the spectral picture. We provide many examples of weighted translation semigroups in multi-variable case. Further, we replace the space $L^2({\mathbb R_+})$ by $L^2({\mathbb R_+^d})$ and explore the properties of weighted translation semigroup $\{S_{\overline{t}}\}$ in ${\cal B}(L^2({\mathbb R_+^d})),$ in both one and multi variable cases. 
\end{abstract}
\subjclass[2010]{Primary 47B20,  47B37; Secondary 47A10, 46E22}
\keywords{weighted translation semigroup, completely hyperexpansive, analytic, Taylor spectrum }
\maketitle
\maketitle

\section{Introduction}
The class of weighted shift operators has been systematically studied in \cite{Se} and generalized to multivariable set up in \cite{JL}. With a view to study a continuous analogue of weighted shifts, M. Embry and A. Lambert initiated the study of a semigroup of operators $\{S_t\}$  indexed by a non-negative real number $t$ in \cite{EL1}, \cite{EL2} and termed it as weighted translation semigroup. The operators $S_t$ are defined on $L^2(\mathbb R_+)$ by using a weight function.
This work is continued in \cite{PS1}, where we characterized some special types of weighted translation semigroups, especially, hyperexpansive weighted translation semigroups. In \cite{PS2}, we proved that the weighted translation semigroup $\{S_t\}$ is analytic and possesses wandering subspace property. Also we proved that a left invertible operator $S_t$ is modeled as a multiplication by $z$ on a suitable reproducing kernel Hilbert space. It turned out that the spectrum of a left invertible operator $S_t$ is a closed disc and the point spectrum is empty. 

In section 2, we introduce the commuting tuple $S_{\bf t}$ whose components are weighted translation semigroups in ${\cal B}(L^2({\mathbb R_+}))$ and discuss its properties. We present several examples in this section.

In section 3, we prove that the tuple $S_{\bf t}$ is analytic and possesses wandering subspace property. We also prove that the tuple $S_{\bf t}$ is unitarily equivalent to a commuting operator valued multishift. Further, we describe the toral analytic model for the toral left invertible tuple $S_{\bf t}$. At the end of this section, the Taylor spectrum of the toral left invertible tuple $S_{\bf t}$ is discussed.
The major part of the work in section 3 is devoted to the comparison of the commuting $d$-tuple $S_{\bf t}$ under consideration and a multishift $S_{\bf \lambda}$ on directed cartesian product of rooted directed trees as described in \cite{CPT}. At the end of this paper, a comparative analysis of these two classes of operator tuples has been summarized.
The techniques used in the proofs of results in section 3 are similar to those in \cite{CPT}. However, the intrinsic differences in the classes of operators get reflected in the proofs accounting for some subtle differences. In the recent work in \cite{Ch}, S. Chavan developed toral and spherical analytic models for a commuting tuple of operators.

In section 4, we define the weighted translation semigroup $\{S_{\overline{t}}\}$ in ${\cal B}(L^2({\mathbb R_+^d}))$ and discuss its properties. In fact, the weighted translation semigroup $\{S_t\}$ is a special case where $d=1.$  
In section 5, we introduce the special commuting tuple whose components are weighted translation semigroups in ${\cal B}(L^2({\mathbb R_+^d}))$ and discuss several properties. In this case, we observe that the symbols of the semigroups under consideration are multi-variable functions. This consideration results in some differences in properties of weighted translation semigroups studied earlier.  

\section{The Commuting tuple}
\subsection{Prelude}
Let $T=(T_1,\cdots ,T_d)$ be a tuple of commuting bounded linear operators $T_i(1\leq i\leq d)$ on a Hilbert space $H.$ Then $T^*$ represents $(T_1^*,\cdots ,T_d^*)$ and for $p=(p_1,\cdots ,p_d)\in \mathbb N^d,~ T^p$ denotes $T_1^{p_1}\cdots T_d^{p_d}.$ 
The open polydisc centered at the origin and of polyradius $r=(r_1,\cdots ,r_d)$ with $r_1,\cdots ,r_d >0$ denoted by $\mathbb D_r^d$ is
$\mathbb D_r^d:=\{z=(z_1,\cdots ,z_d)\in \mathbb C^d~:~|z_1|<r_1,\cdots ,|z_d|<r_d \}.$
The class of completely hyperexpansive operators is explored in \cite{SA} and generalised to multi variable set up in \cite{At}. We first recall some definitions useful in the sequel (\cite{At-1},\cite{CS1}).
\begin{definition} \
\begin{enumerate}
\item A $d$-tuple $S=(S_1,\cdots ,S_d)$ of commuting operators $S_i$ in ${\cal B}(H)$ is subnormal if there exist a Hilbert space $K$ containing $H$ and a $d$-tuple $N=(N_1,\cdots ,N_d)$ of commuting normal operators $N_i$ in ${\cal B}(K)$ such that $N_iH\subseteq H$ and $N_i|H=S_i$ for $1\leq i\leq d.$

\item Let $H_1,\cdots ,H_d $ be commuting operators on $H.$ The tuple $(H_1,\cdots ,H_d)$ is called {\it hyponormal} if the $d\times d$ operator matrix $([H_j^*,H_i])\geq 0,$ where $$[H_j^*,H_i]=H_j^*H_i-H_iH_j^*.$$

\item A commuting $d$-tuple $Q=(Q_1,\cdots ,Q_d)$ of positive operators $Q_1,\cdots ,Q_d$ in ${\cal B}(H)$ is called as the {\it generating $d$-tuple} on $H.$

\item For a commuting $d$-tuple $T=(T_1,\cdots ,T_d)$ of operators on $H,$ the {\it spherical generating $1$-tuple associated with $T$} is given by 
$$Q_s(X)=\sum_{i=1}^d T_i^*XT_i ~~~(X\in {\cal B}(H)).$$

\item For a fixed integer $p\geq 1,$ the tuple $T$ is said to be {\it spherical $p$-expansion (resp. spherical $p$-contraction)} if 
$$B_p(Q_s)=\sum_{q\in \mathbb N,0\leq q\leq p} (-1)^q {p\choose q} Q_s^q(I) \leq 0 ~(resp. \geq 0),$$
where $Q_s^0(I)=I.$

\item The tuple $T$ is called {\it spherical $p$-hyperexpansion (resp. spherical $p$-hypercontraction)} if $T$ is a spherical $k$-expansion (resp. spherical $k$-contraction) for all $k=1,\cdots ,p.$ 
If $B_p(Q_s)=0,$ then $T$ is a {\it spherical $p$-isometry.}

\item The tuple $T$ is said to be {\it spherical complete hyperexpansion (resp. spherical complete hypercontraction)} if $T$ is a spherical $p$-expansion (resp. spherical $p$-contraction) for all positive integers $p.$ 

\item Let $Q_s$ be the spherical generating $1$-tuple associated with $T.$ Then $T$ is {\it jointly left-invertible} if there exists $\alpha >0$ such that $Q_s(I)\geq \alpha I.$ Let $T$ be a jointly left-invertible $d$-tuple of bounded operators on $H.$ The {\it spherical Cauchy dual} of $T$ is defined as the $d$-tuple $T^s=(T_1^s,\cdots ,T_d^s),$ where $T_i^s=T_i(Q_s(I))^{-1}(i=1,\cdots ,d).$ 

\item Given a commuting $d$-tuple $T=(T_1,\cdots ,T_d)$ on $H,$ a {\it toral generating $d$-tuple} is given by $$Q_t=(Q_1,\cdots ,Q_d),~Q_i(X)=T_i^*XT_i (X\in {\cal B}(H)).$$

\item Let $Q_t$ be the toral generating $d$-tuple associated with $T.$ Then $T$ is {\it toral left-invertible} if there exists $\alpha >0$ such that $Q_i(I)\geq \alpha I,$ for $1\leq i\leq d.$ Let $T$ be a toral left-invertible $d$-tuple of bounded operators on $H.$ The {\it toral Cauchy dual} of $T$ is defined as the $d$-tuple $T^\prime=(T_1^\prime,\cdots ,T_d^\prime),$ where $T_i^\prime=T_i(T_i^*T_i)^{-1}(i=1,\cdots ,d).$ 

\item The tuple $T$ is called {\it toral complete hyperexpansion} if $$B_n(Q_t)=\sum_{p\in \mathbb N^d,0\leq p\leq n} (-1)^{|p|} {n\choose p} Q_t^p(I) \leq 0  ~~~\rm{for~ all} ~n\in \mathbb N^d \ ~ \{0\},$$ where $Q_t^q(I)=(Q_1^{q_1}\circ \cdots \circ Q_d^{q_d})(I)$ for $q=(q_1,\cdots ,q_d)\in \mathbb N^d.$ 
\end{enumerate}
\end{definition}
In definitions (5), (6) and (7), if $p=1$ then by convention, the prefix $1$ is dropped and if $m=1$ then we drop the term spherical.

We now define the operator $S_t$ as defined in \cite{PS1}, but in a notation suitable to multi-variable set up.
For a fixed positive integer $i,$ let $\varphi_i$ be a measurable, positive function on $\mathbb R_+$ such that for each fixed $t_i\in \mathbb R_+,$ the function ${\varphi_i}_{t_i}$ defined by 
\begin{equation*}
{\varphi_i}_{t_i}(x) =
\begin{cases}
\displaystyle \sqrt{\frac{\varphi_i(x)}{\varphi_i(x-t_i)}} & \text{if~ $x\geq t_i$},\\
0 & \text{if~ $x < t_i$}
\end{cases}
\end{equation*}
is essentially bounded. 
\begin{definition}
For a fixed positive integer $i$ and each fixed $t_i\in \mathbb R_+,$ we define ${S_i}_{t_i}$ on $L^2({\mathbb R_+})$ by  
\begin{equation*}
{S_i}_{t_i}f(x) =
\begin{cases}
{\varphi_i}_{t_i}(x)f(x-t_i) & \text{if~ $x \geq t_i$},\\
0 & \text{if~ $x < t_i$.}
\end{cases}
\end{equation*}
\end{definition}

The family $G_i=\{{S_i}_{t_i}:t_i\in \mathbb R_+\}$ in ${\cal B}(L^2({\mathbb R_+}))$ is a semigroup with ${S_i}_0=I,$ the identity operator and for all $t_i,r_i~\in \mathbb R_+$, ${S_i}_{t_i}\circ {S_i}_{r_i}={S_i}_{t_i+r_i}.$ Observe that the family $G_1\times \cdots \times G_d$ is also a semigroup. 

We say that ${\varphi_i}_{t_i}$ is a {\it weight function corresponding to the operator ${S_i}_{t_i}$}. Further, the semigroup $G_i=\{{S_i}_{t_i}:t_i\in \mathbb R_+\}$ is referred to as the {\it weighted translation semigroup with symbol $\varphi_i$}. Throughout this article, we assume that the symbol $\varphi_i,1\leq i\leq d$ is a continuous function on $\mathbb R_+.$ 

\begin{remark}
Consider a tuple of operators $({S_1}_{t_1},\cdots ,{S_d}_{t_d}).$
It can be seen that for $1\leq i , j \leq d,$
$${S_i}_{t_i}{S_j}_{t_j}={S_j}_{t_j}{S_i}_{t_i}$$ if and only if 
$$\frac{\varphi_i (x)\varphi_j (x-t_i)}{\varphi_i (x-t_i)\varphi_j (x-t_i-t_j)}
=\frac{\varphi_j (x)\varphi_i (x-t_j)}{\varphi_j (x-t_j)\varphi_i (x-t_i-t_j)}$$
for all $x\geq t_i+t_j.$ 
Observe that if $\varphi_i=\varphi_j,$ then the operators ${S_i}_{t_i}$ and ${S_j}_{t_j}$ commute. \end{remark}
\begin{example} The pair $({S_1}_{t_1},{S_2}_{t_2})$ is commuting for the following symbols:
\begin{enumerate}
\item $\varphi_1(x)=c,~~\varphi_2(x)=e^x$
\item $\varphi_1(x,y)=e^{-x},~~\varphi_2(x,y)=e^x$ \eop
\end{enumerate}
\end{example}
Recall that $${S_i}_{t_i}^*f(x) = \sqrt{\frac{\varphi_i(x+t_i)}{\varphi_i(x)}}f(x+t_i) \quad \mbox{for almost every}~ x\in \mathbb R_+.$$  
Also $${S_i}_{t_i}^*{S_i}_{t_i}f(x) = \frac{\varphi_i(x+t_i)}{\varphi_i(x)}f(x) \quad \mbox{for almost every}~ x\in \mathbb R_+.$$   
Note that the semigroup $G_1\times \cdots \times G_d$ is a toral isometry if every commuting $d$-tuple $S_{\bf t}=({S_1}_{t_1},\cdots ,{S_d}_{t_d})$ in $G_1\times \cdots \times G_d$ is a toral isometry. 
\begin{remark} \label{m4} 
The commuting $d$-tuple $({S_1}_{t_1},\cdots ,{S_d}_{t_d})$ is a toral isometry if and only if $$I-{S_i}_{t_i}^*{S_i}_{t_i}=0 ~{\rm for~ all}~ 1\leq i\leq d.$$ This is true if and only if $$\varphi_i(x)-\varphi_i(x+t_i)=0 ~{\rm for ~all}~ x,t_i\in \mathbb R_+,~1\leq i\leq d.$$ Therefore the semigroup $G_1\times \cdots \times G_d$ is a toral isometry if and only if $\varphi_i$ is a constant function for $1\leq i\leq d.$ \end{remark}

Using the characterization of a toral isometry as stated above, we prove that the toral left invertible commuting $d$-tuple $S_{\bf t}$ admits a polar decomposition.
\begin{proposition} \label{m50} If $S_{\bf t}=({S_1}_{t_1},\cdots ,{S_d}_{t_d})$ is a toral left invertible commuting $d$-tuple, then there exist a toral isometry $({U_1}_{t_1},\cdots ,{U_d}_{t_d})$ and a commuting $d$-tuple $({D_1}_{t_1},\cdots ,{D_d}_{t_d})$ of diagonal, positive, invertible bounded operators on $L^2(\mathbb R_+)$ such that ${S_i}_{t_i}={U_i}_{t_i}{D_i}_{t_i},~~~~1\leq i\leq d.$ Further, this decomposition is unique. \end{proposition}
\begin{proof} 
For a fixed positive integer $i,~1\leq i\leq d,$ define the operators ${U_i}_{t_i}$ and ${D_i}_{t_i}$ on $L^2(\mathbb R_+)$ as follows:
\begin{equation*}
{U_i}_{t_i}f(x) =
\begin{cases}
f(x-t_i) & \text{if~ $x \geq t_i$},\\
0 & \text{if~ $x < t_i$}
\end{cases}
\end{equation*} and 
$$ {D_i}_{t_i}f(x) = \sqrt{\frac{\varphi_i(x+t_i)}{\varphi_i(x)}}f(x) ,\forall x \geq 0. $$ It can be seen that ${S_i}_{t_i}={U_i}_{t_i}{D_i}_{t_i},~~~~1\leq i\leq d$ and this decomposition is unique. \end{proof}
We now quote the proposition which is useful in constructing examples of spherical hyperexpansion using toral hyperexpansion \cite[Proposition 3.7]{CS1}.
\begin{proposition} \label{m5} Let $T=(T_1,\cdots ,T_m)$ be an m-tuple on a Hilbert space $\cal H$ and set $S=(T_1/{\sqrt m},\cdots ,T_m/{\sqrt m}).$ If $T$ is a toral complete hyperexpansion (resp. toral p-expansion, toral p-isometry) then $S$ is a spherical complete hyperexpansion (resp. spherical p-expansion, spherical p-isometry). \end{proposition}
Note that if $T$ is a complete hyperexpansion (resp. $p$-isometry) then the $m$-tuple $(T,\cdots ,T)$ is a toral complete hyperexpansion (resp. toral $p$-isometry). By Propostion \ref{m5}, $(T/\sqrt m,\cdots ,T/\sqrt m)$ is a spherical complete hyperexpansion (resp. spherical $p$-isometry). \\ 
We here quote proposition useful in constructing hyponormal tuples.
\begin{proposition} \cite[Proposition 3]{At-1} \label{m7} Let $S\in \cal B(H).$ Then $S$ is subnormal if and only if $(I,S,\cdots ,S^{p-1})$ is hyponormal for every $p\geq 1.$ \end{proposition}
\begin{example} We here present examples of some special tuples.   
\begin{enumerate}
\item Recall that functions $\varphi_1 (x)=\log (x+2),\varphi_2(x)=\frac{x+\lambda}{x+1} (0< \lambda <1), $ \\ $\varphi_3 (x)=2-e^{-x} $ are completely alternating. Therefore each operator ${S_i}_{t_i}$ in a weighted translation semigroup $G_i$ with symbol $\varphi_i,~~1\leq i\leq 3$ is completely hyperexpansive. Hence the pair $({S_i}_{t_i},{S_i}_{t_i})$ is a toral complete hyperexpansion and the pair $({S_i}_{t_i}/\sqrt 2,{S_i}_{t_i}/\sqrt 2)$ is a spherical complete hyperexpansion for $1\leq i\leq 3.$ 

\item Let $\varphi_1(x)=\sqrt{x+1}.$ Since the function $\varphi_1$ is concave, each operator ${S_1}_{t_1}$ in a weighted translation semigroup $G_1$ with symbol $\varphi_1$ is 2-hyperexpansion. Hence the pair $({S_1}_{t_1},{S_1}_{t_1})$ is a toral 2-hyperexpansion and the pair \\ $({S_1}_{t_1}/\sqrt 2,{S_1}_{t_1}/\sqrt 2)$ is a spherical 2-hyperexpansion. 

\item Observe that each operator ${S_i}_{t_i}$ in a weighted translation semigroup $G_i$ with symbol $\varphi_i,~~1\leq i\leq 2$ is an isometry, where $\varphi_1(x)=c,\varphi_2(x)=d.$ Therefore the pair $({S_1}_{t_1},{S_2}_{t_2}),$ is a toral isometry and the pair \\ $({S_1}_{t_1}/\sqrt 2,{S_2}_{t_2}/\sqrt 2)$ is a spherical isometry.

\item Observe that each operator ${S_1}_{t_1}$ in a weighted translation semigroup $G_1$ with symbol $\varphi_1 (x,y)=x+1$ is a 2-isometry. Hence the pair $({S_1}_{t_1},{S_1}_{t_1})$ is a toral 2-isometry and the pair $({S_1}_{t_1}/\sqrt 2,{S_1}_{t_1}/\sqrt 2)$ is a spherical 2-isometry. 

\item Recall that functions $\varphi_1 (x)=\frac{1}{x+1},\varphi_2(x)=\frac{x+\lambda}{x+1} (\lambda >1), \varphi_3 (x)=e^{-x} $ are completely monotone. Therefore each operator ${S_i}_{t_i}$ in a weighted translation semigroup $G_i$ with symbol $\varphi_i,~~1\leq i\leq 3$ is subnormal contraction. Hence the pair $(I,{S_i}_{t_i})$ is hyponormal. \eop
\end{enumerate}
\end{example} 

\section{Analytic Model and Taylor Spectrum}
In this section, we prove that a commuting $d$-tuple  $S_{\bf t}=({S_1}_{t_1},\cdots ,{S_d}_{t_d})$ is analytic and possesses wandering subspace property. Also a toral left invertible commuting $d$-tuple  $S_{\bf t}$ is modeled as a multiplication by coordinate functions $z_i$ on a suitable reproducing kernel Hilbert space. 

Recall that the Cauchy dual $S_t^{\prime}$ of a left invertible operator $S_t$ is given by  
\begin{equation*}
S_t^{\prime}f(x) =
\begin{cases}
\displaystyle \frac{1}{\varphi_t(x)}f(x-t) & \text{if~ $x\geq t$},\\
0 & \text{if~ $x<t$}.
\end{cases}
\end{equation*}

Let $S_{\bf t}=({S_1}_{t_1},\cdots ,{S_d}_{t_d})$ be a toral left invertible commuting $d$-tuple. Then the toral Cauchy dual tuple $S_{\bf t}^\prime=({S_1}_{t_1}^\prime,\cdots ,{S_d}_{t_d}^\prime)$ of $S_{\bf t}$ is given by for $1\leq i\leq d,$
\begin{equation*}
{S_i}_{t_i}^\prime f(x) =
\begin{cases}
\displaystyle \frac{1}{{\varphi_i}_{t_i}(x)}f(x-t_i) & \text{if~ $x\geq t_i$},\\
0 & \text{if~ $x < t_i$.}
\end{cases}
\end{equation*}
Note that the toral Cauchy dual $d$-tuple $({S_1}_{t_1}^\prime,\cdots ,{S_d}_{t_d}^\prime)$ is commuting if and only if the toral left invertible $d$-tuple $({S_1}_{t_1},\cdots ,{S_d}_{t_d})$ is commuting.  

We first describe the joint kernel $E$ of the commuting $d$-tuple $S_{\bf t}^*=({S_1}_{t_1}^*, \cdots ,{S_d}_{t_d}^*).$ Recall that $E=\cap_{i=1}^d ker {S_i}_{t_i}^*$ and ker ${S_i}_{t_i}^*=\chi_{[0,t_i)}L^2(\mathbb R_+)$ (Refer \cite[Lemma 3.2]{PS2}).  
Observe that $\chi_{[0,t_i)}L^2(\mathbb R_+)\subset \chi_{[0,t_j)}L^2(\mathbb R_+) ~{\rm if}~ t_i< t_j.$ Thus 
 $E=\chi_{[0,t_j)}L^2(\mathbb R_+), $ \\ $~{\rm where}~ t_j=\min \{t_1,\cdots ,t_d\}.$

A commuting $d$-tuple $T=(T_1,\cdots ,T_d)$ on a Hilbert space $H$ is called {\it analytic} if $\cap_{\alpha \in \mathbb N^d}T^\alpha H=\{0\}.$
\begin{proposition} A commuting $d$-tuple $S_{\bf t}=({S_1}_{t_1},\cdots ,{S_d}_{t_d})$ is analytic. \end{proposition}
\begin{proof} Suppose $k=(k_1,\cdots ,k_d )\in \mathbb N^d.$
For simplicity, we use notation $$\sqrt{\frac{\varphi(x)}{\varphi(x-kt)}}=\sqrt{\frac{\varphi_d(x)}{\varphi_d(x-k_dt_d)}}\cdots \sqrt{\frac{\varphi_1(x-k_2t_2-\cdots -k_dt_d)}{\varphi_1(x-k_1t_1-\cdots -k_dt_d)}}.$$
Observe that for $f\in L^2(\mathbb R_+)$
$$ ({S_d}_{t_d}^{k_d}\cdots {S_1}_{t_1}^{k_1}f)(x) = \sqrt{\frac{\varphi(x)}{\varphi(x-kt)}} f(x-k_1t_1-\cdots -k_dt_d) $$
if $x\geq k_1t_1+\cdots +k_dt_d$ and 
$$ ({S_1}_{t_1}^{k_1}\cdots {S_d}_{t_d}^{k_d}f)(x) = 0  ~{\rm if}~ x< k_1t_1+\cdots +k_dt_d.$$
Hence $$\bigcap_{(k_1,\cdots ,k_d )\in \mathbb N^d} {S_1}_{t_1}^{k_1}\cdots {S_d}_{t_d}^{k_d} L^2(\mathbb R_+)\subseteq \bigcap_{(k_1,\cdots ,k_d )\in \mathbb N^d}\chi_{[k_1t_1+\cdots +k_dt_d,\infty)}L^2(\mathbb R_+).$$ The result follows from the fact that 
 $$\bigcap_{(k_1,\cdots ,k_d )\in \mathbb N^d}\chi_{[k_1t_1+\cdots +k_dt_d,\infty)}L^2(\mathbb R_+)=\{0\}.$$ \end{proof}
The following Lemma follows in similar manner as \cite[Lemma 3.5]{PS2}.
\begin{lemma} \label{m8} Let $S_{\bf t}=({S_1}_{t_1},\cdots ,{S_d}_{t_d})$ be a toral left invertible commuting $d$-tuple in $G_1\times \cdots \times G_d$ and let $E$ be the joint kernel of $S_{\bf t}^*.$ Then the multisequence $\{S_{\bf t}^k E\}_{k \in \mathbb N^d}$ is mutually orthogonal. \end{lemma}
\begin{remark} \label{m9} Using similar argument, it can be proved that the multisequence $\{{S_{\bf t}^\prime}^k E\}_{k \in \mathbb N^d}$ is also mutually orthogonal. \end{remark}
We say that a commuting $d$-tuple $T=(T_1,\cdots ,T_d)$ on a Hilbert space $H$ possesses wandering subspace property if $H=[E]_T,$ where $E$ is the joint kernel of $T^*$ and $[E]_T=\vee_{\alpha \in \mathbb N^d}T^\alpha E.$
The following Proposition follows in similar manner as \cite[Proposition 3.4]{PS2}.
\begin{proposition} \label{m10} A commuting $d$-tuple $S_{\bf t}=({S_1}_{t_1},\cdots ,{S_d}_{t_d})$ in $G_1\times \cdots \times G_d$ possesses wandering subspace property. \end{proposition}
\begin{remark} \label{m25} Using same technique given in the proof of \cite[Proposition 3.4]{PS2}, it can be proved that the commuting $d$-tuple $S_{\bf t}^\prime$ also possesses wandering subspace property. \end{remark}

We first define operator valued multishift. 
Let $M$ be a nonzero complex Hilbert space and let $l^2_M(\mathbb N^d)$ denote the Hilbert space of square summable multisequence $\{h_{\alpha}\}_{\alpha \in \mathbb N^d}$ in $M.$ If $\{W_{\alpha}^{(j)}\}_{\alpha \in \mathbb N^d}\subseteq \cal B(M)$ for $j=1,\cdots ,d,$ then the linear operator $W_j$ in $l^2_M(\mathbb N^d)$ is defined by $W_j(h_\alpha)_{\alpha \in \mathbb N^d}=(k_\alpha)_{\alpha \in \mathbb N^d}$ for $(h_\alpha)_{\alpha \in \mathbb N^d}\in \cal D,$ where 
\begin{equation*}
 (k_\alpha)=
\begin{cases}
W_{\alpha-\epsilon_j}^{(j)}h_{\alpha-\epsilon_j} & \text{if~ $\alpha_j\geq 1$},\\
0 & \text{if~ $\alpha_j =0$}
\end{cases}
\end{equation*}
and ${\cal D}:=\{(h_\alpha)_{\alpha \in \mathbb N^d}\in l^2_M(\mathbb N^d)~:~(k_\alpha)_{\alpha \in \mathbb N^d}\in l^2_M(\mathbb N^d)\}.$ \\
The $d$-tuple $W=(W_1,\cdots ,W_d)$ is called {\it an operator valued multishift with operator weights $\{W_{\alpha}^{(j)}~:~\alpha \in \mathbb N^d,j=1,\cdots ,d\}$.}

The proof of the following proposition follows in a similar way as given in \cite[Corollary 4.1.12]{CPT}.
\begin{proposition} \label{m11} Let $S_{\bf t}=({S_1}_{t_1},\cdots ,{S_d}_{t_d})$ be a toral left invertible commuting $d$-tuple in $G_1\times \cdots \times G_d.$ Let $E$ be the joint kernel of $S_{\bf t}^*.$ Then $S_{\bf t}$ is unitarily equivalent to a commuting operator valued multishift $W$ on $l^2_E(\mathbb N^d).$ \end{proposition}

We now define the kernel condition useful in constructing the analytic model.
\begin{definition} Let $T=(T_1,\cdots ,T_d)$ be a toral left invertible $d$-tuple on a Hilbert space $H$ and let $E$ denote the joint kernel of $T^*.$ Let $T^\prime$ be the toral Cauchy dual of $T.$ We say that $T$ satisfies {\it kernel condition} if $E\subseteq \rm{ker} T_j^*{T^\prime}^\alpha_{[j]}$ for all $j=1,\cdots ,d$ and for all $\alpha \in \mathbb N^d,$ where,   for $\alpha \in \mathbb N^d$ and $j=1,\cdots ,d,$ 
\begin{equation*}
{T^\prime}^\alpha_{[j]} =
\begin{cases}
I & \text{if~ $d=1$},\\
\prod_{i\neq j}{T_i^\prime}^{\alpha_i}  & \text{if~ $d \geq 2.$}
\end{cases}
\end{equation*}
\end{definition}

Note that the toral left invertible commuting $d$-tuple $S_{\bf t}=({S_1}_{t_1},\cdots ,{S_d}_{t_d})$ satisfies kernel condition \cite[Remark 4.2.3(iii)]{CPT}.

The following theorem is a multivariable analogue of Shimorin's analytic model developed in \cite[Theorem 3.7]{PS2}. The reader may compare this theorem with \cite[Theorem 4.2.4]{CPT} for a similar model developed in the context of multishifts on directed cartesian product of rooted directed trees.
\begin{theorem} \label{m12} Let $S_{\bf t}=({S_1}_{t_1},\cdots ,{S_d}_{t_d})$ be a toral left invertible commuting $d$-tuple in $G_1\times \cdots \times G_d.$ Let $E$ be the joint kernel of $S_{\bf t}^*$ and let $$r:=(r({S_1}_{t_1}^\prime)^{-1},\cdots ,r({S_d}_{t_d}^\prime)^{-1}),$$ where $r(T)$ denotes the spectral radius of the bounded operator $T.$ Then there exist a reproducing kernel Hilbert space $\cal H$ of $E$-valued analytic functions defined on the polydisc $\mathbb D_r^d$ with center origin and polyradius $r$ and a unitary operator $U:L^2(\mathbb R_+)\rightarrow \cal H$ such that $U{S_j}_{t_j}=M_{z_j}U$ for $j=1,\cdots ,d.$  \end{theorem}
\begin{proof} For $f\in L^2(\mathbb R_+),$ define $$U_f(z):=\sum_{k\in \mathbb N^d}(P{S_{\bf t}^\prime }^{*k}f)z^k,~~z\in \mathbb C^d,$$ where $P$ is the orthogonal projection on $E.$ The power series $U_f$ converges absolutely on the polydisc $\mathbb D_r^d.$ Let $\cal H$ denote the complex vector space of $E$-valued analytic functions of the form $U_f.$ We now define a map $U:L^2(\mathbb R_+)\rightarrow \cal H$ given by $U(f)=U_f.$ By definition, $U$ is onto. We now need to show that $U$ is injective. 

Let $U(f)=U_f=0$ for some $f\in L^2(\mathbb R_+).$ Then $\sum_{k\in \mathbb N^d}(P{S_{\bf t}^\prime }^{*k}f)z^k=0.$ This implies that $P{S_{\bf t}^\prime }^{*k}f=0$ for all $k\in \mathbb N^d.$ The fact ${\rm ker}~{S_j}_{t_j}^{\prime *}={\rm ker}~{S_j}_{t_j}^*,$ implies that the joint kernel of the tuple ${S_{\bf t}^\prime *}=E.$ 
By the wandering subspace property of the $d$-tuple $S_{\bf t}^\prime$ [Remark \ref{m25}], we have $\bigvee_{k\in \mathbb N^d} {S_{\bf t}^\prime }^{k}= L^2(\mathbb R_+).$
By taking orthogonal complement on both sides, we get $\bigcap_{k\in \mathbb N^d} ({S_{\bf t}^\prime }^{k}(E))^\bot =\{0\}.$ Note that $({S_{\bf t}^\prime }^{k}(E))^\bot = {\rm ker}~P{S_{\bf t}^\prime }^{*k}$ for every $k\in \mathbb N^d.$ Thus $\bigcap_{k\in \mathbb N^d} {\rm ker}~P{S_{\bf t}^\prime }^{*k}=\{0\}.$ Recall that $P{S_{\bf t}^\prime }^{*k}f=0$ for all $k\in \mathbb N^d.$ Therefore $f \in \bigcap_{k\in \mathbb N^d} {\rm ker}~P{S_{\bf t}^\prime }^{*k}.$ This implies that $f=0.$ Hence $U$ is injective. 

The inner product on $\cal H$ can be now defined as $\left\langle U_f,U_g \right\rangle_{\cal H}=\left\langle f,g \right\rangle_{L^2(\mathbb R_+)}$ for all $f,g \in L^2(\mathbb R_+).$ Since $U$ is a sujective isometry, it is a unitary operator. A complete vector space $\cal H$ is now a Hilbert space. For $f\in L^2(\mathbb R_+),$
\beqn 
(U{S_j}_{t_j})(f)(z) &=& \sum_{k\in \mathbb N^d} (P{S_{\bf t}^\prime }^{*k}{S_j}_{t_j}f)z^k \\
&=& \sum_{k\in \mathbb N^d,k_j=0} (P{S_{\bf t}^\prime }^{*k}{S_j}_{t_j}f)z^k + \sum_{k\in \mathbb N^d,k_j\geq 1} (P{S_{\bf t}^\prime }^{*k}{S_j}_{t_j}f)z^k \\
&=& \sum_{k\in \mathbb N^d,k_j=0} (P{S_{{\bf t}[j]}^\prime }^{*k}{S_j}_{t_j}f)z^k + \sum_{k\in \mathbb N^d} (P{S_{\bf t}^\prime }^{*k+\epsilon_j}{S_j}_{t_j}f)z^{k+\epsilon_j} \\
&=& \sum_{k\in \mathbb N^d} (P{S_{\bf t}^\prime }^{*k+\epsilon_j}{S_j}_{t_j}f)z^{k+\epsilon_j}. 
\eeqn
Note that by kernel condition, $\sum_{k\in \mathbb N^d,k_j=0} (P{S_{{\bf t}[j]}^\prime }^{*k}{S_j}_{t_j}f)z^k =0.$
Since the toral Cauchy dual tuple $S_{\bf t}^\prime$ is commuting and ${S_j}_{t_j}^{\prime *}{S_j}_{t_j}=I,$ the sum on the right hand side of the last step is equal to $z_j \sum_{k\in \mathbb N^d} (P{S_{\bf t}^\prime }^{*k}f)z^k =z_jU_f(z)=M_{z_j}U(f)(z).$ 
Hence $U{S_j}_{t_j}=M_{z_j}U.$

Using Remark \ref{m9} and argument similar to \cite[Theorem 3.7(ii)]{PS2}, it can be proved that $(P{S_{\bf t}^\prime}^{*j}{S_{\bf t}^\prime}^k)|_E=0$ for all $j\neq k$ in $\mathbb N^d.$
We now compute the reproducing kernel for the RKHS $\cal H$ associated to the operator tuple $S_{\bf t}.$ These computations are similar as given in the proof of \cite[Proposition 2.13]{Sh}. 
Note that $E=$ker $S_{\bf t}^*$ is infinite dimensional. It is easy to see that for $e\in E,~z,\lambda \in \mathbb D_r^d$ and $x\in \mathbb R_+,~ \displaystyle
 k_{\cal H}(z,\lambda)e(x) 
= \left( \sum_{n \in \mathbb N^d} P{S_{\bf t}^\prime}^{*n}{S_{\bf t}^\prime}^n z^n\overline{\lambda}^n \right) e(x).$
It can be seen that for $f,g\in E$ and $\lambda \in \mathbb D_r^d,$ 
$\left\langle U_f, k_{\cal H}(.,\lambda)g \right\rangle_{\cal H}=\left\langle U_f(\lambda),g \right\rangle_{E}.$
Hence $\cal H$ is a reproducing kernel Hilbert space with kernel $k.$ 
\end{proof}


\begin{example} We now compute the reproducing kernel for the RKHS $\cal H$ associated to the pair $({S_1}_{t_1},{S_2}_{t_2})$ in some special cases. 
\begin{enumerate} 
\item Let $\varphi_1(x)=a,\varphi_2(x)=b.$ Then the pair $({S_1}_{t_1},{S_2}_{t_2})$ is a toral isometry. 
For $e\in E,~z,\lambda \in \mathbb D_r^2$ and $x \in \mathbb R_+,$
$$  k_{\cal H}(z,\lambda)e(x) 
= \left( \sum_{n \in \mathbb N^2} P{S_{\bf t}^\prime}^{*n}{S_{\bf t}^\prime}^n z^n\overline{\lambda}^n \right) e(x) 
= \left( \sum_{(n_1,n_2) \in \mathbb N^2} z_1^{n_1}z_2^{n_2}\overline{\lambda_1}^{n_1}\overline{\lambda_2}^{n_2} \right) e(x). $$
\item Let $\varphi_1(x)=\varphi_2(x).$ For $e\in E,~z,\lambda \in \mathbb D_r^2$ and $x \in \mathbb R_+,$
\beqn
 k_{\cal H}(z,\lambda)e(x) 
&=& \left( \sum_{n \in \mathbb N^2} P{S_{\bf t}^\prime}^{*n}{S_{\bf t}^\prime}^n z^n\overline{\lambda}^n \right) e(x) \\
&=& \left( \sum_{(n_1,n_2) \in \mathbb N^2} \frac{\varphi_1(x)}{\varphi_1(x+n_1t_1+n_2t_2)}z_1^{n_1}z_2^{n_2}\overline{\lambda_1}^{n_1}\overline{\lambda_2}^{n_2} \right) e(x). \eeqn \eop
\end{enumerate}
\end{example}

\subsection{Taylor Spectrum}
For the definition and basic properties of Taylor spectrum, the reader is referred to \cite{Cu}.
\begin{proposition} \label{m13} Let $S_{\bf t}=({S_1}_{t_1},\cdots ,{S_d}_{t_d})$ be a commuting $d$-tuple in $G_1\times \cdots \times G_d.$ For every $\theta \in \mathbb R$ there exists a unitary operator $M_\theta$ on $L^2(\mathbb R_+)$ such that $M_\theta^*{S_j}_{t_j}M_\theta =e^{-i\theta t_j}{S_j}_{t_j}$ for $j=1,\cdots ,d.$ \end{proposition}
\begin{proof} For $\theta \in \mathbb R,$ define the map $M_\theta :L^2(\mathbb R_+)\rightarrow L^2(\mathbb R_+)$ as $$(M_\theta f)(x)=e^{i\theta x}f(x).$$ Then clearly $M_\theta$ is a unitary operator. It can be seen that the operator ${S_j}_{t_j}$ is unitarily equivalent to the operator $e^{-i\theta t_j}{S_j}_{t_j}.$
\end{proof}
\begin{remark} \label{m14} 
By the Proposition \ref{m13}, the commuting $d$-tuples $({S_1}_{t_1},\cdots ,{S_d}_{t_d})$ and $(e^{-i\theta t_1}{S_1}_{t_1},\cdots ,e^{-i\theta t_d}{S_d}_{t_d})$ are unitarily equivalent.
Hence the Taylor spectrum of a commuting $d$-tuple $S_{\bf t}$ has poly-circular symmetry, that is, for any $w=(w_1,\cdots ,w_d) \in \sigma(S_{\bf t})$ and $z=(z_1,\cdots ,z_d) \in \mathbb T^d, z.w=(z_1w_1,\cdots ,z_dw_d) \in \sigma(S_{\bf t}).$ In particular, $\sigma(S_{\bf t}^*)=\sigma(S_{\bf t}).$
\end{remark}

For proving the connectedness of a Taylor spectrum of a commuting $d$-tuple $S_{\bf t},$ we need a lemma.
Given a positive integer $d,$ we set $H^{\oplus d}:=H\oplus \cdots \oplus H$ ($d$ times). For a commuting $d$-tuple $T=(T_1,\cdots ,T_d)$ of operators on a Hilbert space $H,$ consider the linear transformation $D_T~:~H\rightarrow H^{\oplus d}$ given by for $h\in H$ $$D_T(h):=(T_1h,\cdots ,T_dh).$$
Observe that ker $D_T=\rm{ker}~ T=\cap_{j=1}^d \rm{ker}~ T_j.$ 
\begin{lemma} \label{m15} Let $S_{\bf t}=({S_1}_{t_1},\cdots ,{S_d}_{t_d})$ be a commuting $d$-tuple. For $i\in \mathbb N,$ let $T^{(i)}$ denote the commuting $d$-tuple $({{S_1}_{t_1}^*}^i,\cdots ,{{S_d}_{t_d}^*}^i).$ Then $\cup_{i\in \mathbb N}\rm{ker}~ D_{T^{(i)}}$ is dense in $L^2(\mathbb R_+).$ \end{lemma}
\begin{proof} 
Note that ker $D_{T^{(i)}}=\chi_{[0,it_r)}L^2(\mathbb R_+),$ where $t_r=\min \{t_1,\cdots ,t_d\}.$
It is sufficient to prove that $\cup_{i\in \mathbb N}\rm{ker}~ D_{T^{(i)}}$ contains some orthonormal basis of $L^2(\mathbb R_+).$ 
Let $\{\widetilde{\psi_{jk}}\},$ where $j$ is an integer and $k$ is a non-negative integer be an orthonormal basis of $L^2(\mathbb R_+)$ as described in \cite[Theorem 3.7(iv)]{PS2}. For fixed $j$ and $k,$ if $\frac{k+1}{2^j}< it_r,$ then $\widetilde{\psi_{jk}} \in \rm{ker}~ D_{T^{(i)}}.$ Therefore each $\widetilde{\psi_{jk}} \in \rm{ker}~ D_{T^{(i)}}$ for some non-negative integer $i.$ Hence the orthonormal basis $\{\widetilde{\psi_{jk}}\}\subseteq \cup_{i\in \mathbb N}\rm{ker} D_{T^{(i)}}.$ \end{proof}

Now proof of the following proposition follows on the lines similar to the proof of \cite[Proposition 3.2.4]{CPT}.
\begin{proposition} \label{m16} The Taylor spectrum of a commuting $d$-tuple $S_{\bf t}=({S_1}_{t_1},\cdots ,{S_d}_{t_d})$ is connected. \end{proposition}

\begin{remark} \label{m17} Recall that the point spectrum of each operator ${S_j}_{t_j}$ is empty \cite[Proposition 4.3(1)]{PS2}. This implies that the point spectrum of a commuting $d$-tuple $S_{\bf t}$ is empty. \end{remark}




\begin{remark} Recall that $\mathbb D_r^d$ is the polydisc with center origin and polyradius $r,$ where $r=(r({S_1}_{t_1}^\prime)^{-1},\cdots ,r({S_d}_{t_d}^\prime)^{-1}).$ Note that the polydisc $\mathbb D_r^d$ is contained in the point spectrum of $S_{\bf t}^*$ for a toral left invertible commuting $d$-tuple $S_{\bf t}.$ The proof of this fact is similar to the proof of \cite[Proposition 4.3(2)]{PS2}. By the poly-circular symmetry of the Taylor spectrum, $\mathbb D_r^d\subseteq \sigma(S_{\bf t}).$ Suppose $R$ is the polyradius $R=(r({S_1}_{t_1}),\cdots ,r({S_d}_{t_d})).$ Then by the projection property of the Taylor spectrum, $\sigma(S_{\bf t})\subseteq \sigma({S_1}_{t_1})\times \cdots \times \sigma({S_d}_{t_d})=\mathbb D_R^d.$ Also the Taylor spectrum of $S_{\bf t}$ is connected [Proposition \ref{m16}]. 
If the commuting $d$-tuple $S_{\bf t}$ is a toral isometry, then $r=R=(1,\cdots,1).$ In this case, the Taylor spectrum is a polydisc. 
\end{remark}  

\subsection{Comparative Analysis}
We now summarize the comparative analysis of the commuting $d$-tuple $S_{\bf t}$ under consideration and a multishift $S_{\bf \lambda}$ on directed cartesian product of rooted directed trees as described in \cite{CPT}.
\begin{enumerate} 
\item The joint kernel of the tuple $S_{\bf t}^*$ is always infinite dimensional but the joint kernel of $S_{\bf \lambda}^*$ may be finite dimensional \cite[Corollary 3.1.16]{CPT}. 
\item The toral Cauchy dual tuple $S_{\bf t}^\prime$ is always commuting but the toral Cauchy dual tuple $S_{\bf \lambda}^\prime$ may not be commuting \cite[Proposition 5.1.1]{CPT}. 
\item The tuple $S_{\bf t}$ admits polar decomposition. However, the multishift $S_{\bf \lambda}$ admits polar decomposition if and only if the toral Cauchy dual tuple of $S_{\bf \lambda}$ is commuting \cite[Proposition 5.1.1]{CPT}.
\item The multisequence $\{S_{\bf t}^k E\}_{k \in \mathbb N^d}$ is mutually orthogonal but, the multisequence $\{S_{\bf \lambda}^k E\}_{k \in \mathbb N^d}$ may not be mutually orthogonal.
\item The toral left invertible commuting $d$-tuple $S_{\bf t}$ satisfies kernel condition but the multishift $S_{\bf \lambda}$ may not satisfy kernel condition. (refer to paragraph after \cite[Remark 4.2.3]{CPT}).
\item The Taylor spectrum of the tuple $S_{\bf t}$ has poly circular symmetry. However, the operator tuple $S_{\bf \lambda}$ is strongly circular \cite[Proposition 3.2.1]{CPT}.
\end{enumerate}

\section{Weighted Translation semigroups in ${\cal B}(L^2({\mathbb R_+^d}))$}
All the properties of the weighted translation semigroup $\{S_t\}$ in ${\cal B}(L^2({\mathbb R_+}))$ proved in \cite{PS1} and \cite{PS2}, are also shared by the semigroup $\{S_{\overline{t}}\}$ in ${\cal B}(L^2({\mathbb R_+^d})).$  
For positive integer $d,$ let $L^2({\mathbb R_+^d})$ denote the Hilbert space of complex valued square integrable Lebesgue measurable functions on $\mathbb R_+^d.$ Let ${\cal B}(L^2({\mathbb R_+^d}))$ denote the algebra of bounded linear operators on $L^2({\mathbb R_+^d}).$  
\begin{definition}
For a measurable, positive function $\varphi$ defined on $\mathbb R_+^d$ and \\ $\overline{t}=(t_1,\cdots ,t_d),\overline{x}=(x_1,\cdots ,x_d)\in \mathbb R_+^d$ define the function \\ $\varphi_{\overline{t}} : \mathbb R_+^d \rightarrow \mathbb R_+$ by 
\begin{equation*}
\varphi_{\overline{t}}(\overline{x}) =
\begin{cases}
\displaystyle \sqrt {\frac{\varphi(\overline{x})}{\varphi(\overline{x}-\overline{t})}} & \text{if~ $x_i\geq t_i,1\leq i\leq d$},\\
0 & \text{otherwise}.
\end{cases}
\end{equation*}
Here, $\overline{x}-\overline{t}=(x_1-t_1,\cdots ,x_d-t_d).$ \end{definition}

Suppose that $\varphi_{\overline{t}}$ is essentially bounded for every $\overline{t} \in \mathbb R_+^d$.
\begin{definition}
For each fixed $\overline{t}\in \mathbb R_+^d,$ we define $S_{\overline{t}}$ on $L^2({\mathbb R_+^d})$ by  
\begin{equation*}
S_{\overline{t}}f(\overline{x}) =
\begin{cases}
\varphi_{\overline{t}}(\overline{x})f(\overline{x}-\overline{t}) & \text{if~ $x_i\geq t_i,1\leq i\leq d$},\\
0 & \text{otherwise}.
\end{cases}
\end{equation*}
\end{definition}
\begin{remark}
It is easy to see that for every $\overline{t}\in \mathbb R_+^d,~S_{\overline{t}}$ is a bounded linear operator on $L^2({\mathbb R_+^d})$ with $\|S_{\overline{t}}\|=\|\varphi_{\overline{t}}\|_\infty,$  where $\|\varphi_{\overline{t}}\|_\infty$ stands for the essential supremum of $\varphi_{\overline{t}}.$ 
The family $\{S_{\overline{t}}:t\in \mathbb R_+^d\}$ in ${\cal B}(L^2({\mathbb R_+^d}))$ is a semigroup with $S_{\overline{0}}=I,$ the identity operator and for all $\overline{t},\overline{s}~\in \mathbb R_+^d$, $S_{\overline{t}}\circ S_{\overline{s}}=S_{\overline{t+s}}.$ Here, $\overline 0$ is the $d$-tuple $(0,\cdots, 0)$ in $\mathbb R_+^d.$ \end{remark}

We say that $\varphi_{\overline{t}}$ is a {\it weight function corresponding to the operator $S_{\overline{t}}$}. Further, the semigroup $\{S_{\overline{t}}:\overline{t}\in \mathbb R_+^d\}$ is referred to as the {\it weighted translation semigroup with symbol $\varphi$}.
Throughout this work, we assume that the symbol $\varphi$ is a continuous function on $\mathbb R_+^d.$ 

By similar computations as in the case of an operator $S_t$ in ${\cal B}(L^2(\mathbb R_+)),$ it is easy to see that the adjoint of $S_{\overline t}$ is given by 
$$S_{\overline t}^*f({\overline x}) = \sqrt {\frac{\varphi({\overline x}+{\overline t})}{\varphi({\overline x})}}f({\overline x}+{\overline t}) \quad \mbox{for almost every}~ {\overline x} \in \mathbb R_+^d$$ and 
$$S_{\overline t}^*S_{\overline t}f({\overline x}) = \frac{\varphi({\overline x}+{\overline t})}{\varphi({\overline x})}f({\overline x}+{\overline t}) \quad \mbox{for almost every}~ {\overline x} \in \mathbb R_+^d.$$
Here, $\overline{x}+\overline{t}=(x_1+t_1,\cdots ,x_d+t_d).$

We now present the special types of weighed translation semigroups by choosing some special types of symbols. The characterizations of weighed translation semigroups in this case are similar to those as described in \cite[Corollary 3.3]{PS1}. 
\begin{example} Let $\{S_{\overline t}\}$ be a weighted translation semigroup with symbol $\varphi.$
\begin{enumerate}
\item
Let $\varphi_1(x,y)=e^{(x+y)} , \varphi_2(x,y)=e^{\sqrt{xy}}.$
All these functions are log convex. Thus the semigroups $\{S_{\overline t}\}$ corresponding to symbols $\varphi_i,i=1,2$ are hyponormal.
\item 
Let $\varphi_1(x,y)=2x-y-x^2+2xy-y^2+1, \varphi_2(x,y)=Ax^ay^b+1 \\ (A>0,a,b\geq 0, a+b\leq 1).$
All these functions are concave. Thus the semigroups $\{S_{\overline t}\}$ corresponding to symbols $\varphi_i,i=1,2$ are 2-hyperexpansive.
\item
The semigroup $\{S_{\overline t}\}$ is an $m$-isometry if $\phi$ is a polynomial of degree $m-1.$
\item
Let $\varphi_1(x,y)= \log (x+y+2), \varphi_2(x,y)=\sqrt{(x+1)(y+1)}, \\ \varphi_3(x,y)=x+y+1.$
All these functions are completely alternating.
Thus the semigroups $\{S_{\overline t}\}$ corresponding to symbols $\varphi_i,i=1,2,3$ are completely hyperexpansive.
\item
Let $\varphi_1(x,y)=\frac{1}{x+y+1},~ \varphi_2(x,y)=\frac{1}{\sqrt{x+y+1}},~ \varphi_3(x,y)=e^{-(x+y)}.$ All these functions are completely monotone. Thus the semigroups $\{S_{\overline t}\}$ corresponding to symbols $\varphi_i,i=1,2,3$ are subnormal contractions.
\item
Let $ \varphi_1(x,y)=ye^x+xe^y+1 , \varphi_2(x,y)=e^{xy}.$
All these functions are absolutely monotone.
Thus the semigroups $\{S_{\overline t}\}$ corresponding to symbols $\varphi_i,i=1,2$ are alternatingly hyperexpansive.
\item
Let $\varphi (x,y)=\frac{x+y+\lambda}{x+y+1},~\lambda \geq 0.$ 
Then $\frac{\partial ^n \varphi }{\partial x^{k_1}\partial y^{k_2}} =\frac{(1-\lambda)(-1)^{n-1}n!}{(x+y+1)^{n+1}},$ where \\ $n=k_1+k_2.$
If $0< \lambda <1,$ then $\varphi$ is a completely alternating function and the semigroup $\{S_{\overline t}\}$ is completely hyperexpansive.
If $\lambda >1,$ then $\varphi$ is a completely monotone function and the semigroup $\{S_{\overline t}\}$ is a subnormal contraction.
\end{enumerate} \end{example}

We now describe the kernel of the adjoint $S_{\overline t}^*,$ for $\overline t \neq \overline 0.$
\begin{lemma} For $\overline t \neq \overline 0,$ ker $S_{\overline t}^*=E=\chi_{[0,t_1)\times \cdots \times [0,t_d)}L^2(\mathbb R_+^d).$ In particular, ker $S_{\overline t}^*$ is infinite dimensional. \end{lemma}

\begin{remark} \label{m25} The following properties of $S_{\overline t}$ are analogues of the corresponding properties of $S_t.$ Let $\overline t \neq \overline 0.$  
\begin{enumerate}
\item An operator $S_{\overline t}$ is analytic.
\item An operator $S_{\overline t}$ possesses wandering subspace property.
\item A left invertible operator $S_{\overline t}$ possesses an analytic model.
\item A left invertible operator $S_{\overline t}$ is an operator valued weighted shift.
\item The spectrum of a left invertible operator $S_{\overline t}$ is a disc and the point spectrum is empty.
\end{enumerate}
\end{remark}

\section{The Special Tuple}
We now define a tuple of a special type. Recall that a commuting $d$-tuple $S_{\bf t}$ defined in Section 2, is a toral isometry if and only if each symbol $\varphi_i,$ \\ $1\leq i \leq d$ is a constant function [Remark \ref{m4}]. However, in the special case, defined below we do have non-constant functions as symbols for a commuting tuple which is a toral isometry.    

Let $\varphi_i$ be a measurable, positive function on $\mathbb R_+^d$ such that for each fixed $t_i\in \mathbb R_+,$ the function ${\varphi_i}_{t_i}$ defined by 
\begin{equation*}
{\varphi_i}_{t_i}(x_1,\cdots ,x_d) =
\begin{cases}
\displaystyle \sqrt{\frac{\varphi_i(x_1,\cdots ,x_d)}{\varphi_i(x_1,\cdots ,x_i-t_i,\cdots ,x_d)}} & \text{if~ $x_i\geq t_i$},\\
0 & \text{if~ $x_i < t_i$}
\end{cases}
\end{equation*}
is essentially bounded. 
\begin{definition}
For each fixed $t_i\in \mathbb R_+,$ we define ${S_i}_{t_i}$ on $L^2({\mathbb R_+^d})$ by  
\begin{equation*}
{S_i}_{t_i}f(x_1,\cdots ,x_d) =
\begin{cases}
{\varphi_i}_{t_i}(x_1,\cdots ,x_d)f(x_1,\cdots ,x_i-t_i,\cdots ,x_d) & \text{if~ $x_i\geq t_i$}\\
0 & \text{if~ $x_i < t_i$.}
\end{cases}
\end{equation*}
\end{definition}

\begin{remark} \label{m2} 
It is easy to see that for every $t_i\in \mathbb R_+,~{S_i}_{t_i}$ is a bounded linear operator on $L^2({\mathbb R_+^d})$ with $\|{S_i}_{t_i}\|=\|{\varphi_i}_{t_i}\|_\infty,$  where $\|{\varphi_i}_{t_i}\|_\infty$ stands for the essential supremum of ${\varphi_i}_{t_i}.$ 
The family $G_i=\{{S_i}_{t_i}:t_i\in \mathbb R_+\}$ in ${\cal B}(L^2({\mathbb R_+^d}))$ is a semigroup with ${S_i}_0=I,$ the identity operator. Note that the family $G_1\times \cdots \times G_d$ is also a semigroup. \end{remark}

We say that ${\varphi_i}_{t_i}$ is a {\it weight function corresponding to the operator ${S_i}_{t_i}$}. Further, the semigroup $G_i=\{{S_i}_{t_i}:t_i\in \mathbb R_+\}$ is referred to as the {\it weighted translation semigroup with symbol $\varphi_i$}. Throughout this work, we assume that the symbol $\varphi_i,1\leq i\leq d$ is a continuous function on $\mathbb R_+^d.$ 

\begin{remark} \label{m3} 
Observe that ${S_i}_{t_i}=S_{\overline t},$ where $\overline t=(0,\cdots ,0,t_i,0, \cdots ,0).$ Here, the $i^{th}$ component of the $d$-tuple $\overline t$ is $t_i$ and all other entries are zero. 
Therefore by Remark \ref{m25}, each operator ${S_i}_{t_i}$ is analytic and possesses wandering subspace property. Also the spectrum of each left invertible operator ${S_i}_{t_i}$ is a closed disc and the point spectrum is empty. \end{remark}

Consider a tuple $({S_1}_{t_1},\cdots ,{S_d}_{t_d})\in G_1\times \cdots \times G_d.$
It is easy to see that for $1\leq i,j \leq d,$
$${S_i}_{t_i}{S_j}_{t_j}={S_j}_{t_j}{S_i}_{t_i}$$ if and only if 
$$\frac{\varphi_i (x_1,\cdots ,x_d)\varphi_j (x_1,\cdots ,x_i-t_i,\cdots ,x_d)}{\varphi_i (x_1,\cdots , x_i-t_i,\cdots ,x_d)\varphi_j (x_1,\cdots ,x_i-t_i,\cdots ,x_j-t_j, \cdots ,x_d)}$$
$$=\frac{\varphi_j (x_1,\cdots ,x_d)\varphi_i (x_1,\cdots ,x_j-t_j,\cdots ,x_d)}{\varphi_j (x_1,\cdots ,x_j-t_j,\cdots ,x_d)\varphi_i (x_1,\cdots ,x_i-t_i,\cdots ,x_j-t_j, \cdots ,x_d)}$$
for all $x_i\geq t_i,x_j\geq t_j.$ 
Note that if $\varphi_i=\varphi_j,$ then above condition is satisfied.
\begin{example} Let $b,c,\lambda $ be positive real numbers. The pair $({S_1}_{t_1},{S_2}_{t_2})$ is commuting for the following symbols: 
\begin{enumerate}
\item $\varphi_1(x,y)=c,~~\varphi_2(x,y)=b$
\item $\varphi_1(x,y)=c,~~\varphi_2(x,y)=e^{-x-y}$
\item $\varphi_1(x,y)=c,~~\varphi_2(x,y)=e^{x+y}$
\item $\varphi_1(x,y)=e^{-x-y},~~\varphi_2(x,y)=e^{x+y}$
\item $\varphi_1(x,y)=y+1,~~\varphi_2(x,y)=x+1$
\item $\displaystyle \varphi_1(x,y)=\frac{x+\lambda}{x+y+1},~~\varphi_2(x,y)=\frac{y+\lambda}{x+y+1}$ \eop
\end{enumerate}
\end{example}

It is easy to see that $${S_i}_{t_i}^*f(x_1,\cdots ,x_d) = \sqrt{\frac{\varphi_i(x_1,\cdots ,x_i+t_i,\cdots ,x_d)}{\varphi_i(x_1,\cdots , ,x_d)}}f(x_1,\cdots ,x_i+t_i,\cdots ,x_d)$$ for almost every $(x_1,\cdots ,x_d)\in \mathbb R_+^d$ 
and $${S_i}_{t_i}^*{S_i}_{t_i}f(x_1,\cdots ,x_d) = \frac{\varphi_i(x_1,\cdots ,x_i+t_i,\cdots ,x_d)}{\varphi_i(x_1,\cdots , ,x_d)}f(x_1,\cdots ,x_d)$$ for almost every $(x_1,\cdots ,x_d) \in \mathbb R_+^d.$ 

We say that the semigroup $G_1\times \cdots \times G_d$ is a {\it toral isometry} if every commuting $d$-tuple $({S_1}_{t_1},\cdots ,{S_d}_{t_d})$ in $G_1\times \cdots \times G_d$ is a toral isometry.
\begin{remark} 
The commuting $d$-tuple $({S_1}_{t_1},\cdots ,{S_d}_{t_d})$ is a toral isometry if and only if $$I-{S_i}_{t_i}^*{S_i}_{t_i}=0 ~{\rm for~ all}~ 1\leq i\leq d.$$ 
A simple calculation reveals that this condition holds if and only if $$\varphi_i(x_1,\cdots ,x_d)-\varphi_i(x_1,\cdots ,x_i+t_i,\cdots ,x_d)=0 ~~\forall x_1,\cdots ,x_d\in \mathbb R_+.$$
Therefore the semigroup $G_1\times \cdots \times G_d$ is a toral isometry if and only if 
$$\varphi_i(x_1,\cdots ,x_d)-\varphi_i(x_1,\cdots ,x_i+t_i,\cdots ,x_d)=0 ~~\forall x_1,\cdots ,x_d,t_i\in \mathbb R_+.$$
But this condition holds if and only if $\varphi_i$ is independent of $x_i$ for $1\leq i\leq d.$ \end{remark}
\begin{remark} The toral left invertible commuting $d$-tuple $({S_1}_{t_1},\cdots ,{S_d}_{t_d})$ also admits polar decomposition as given in the Proposition 2.6. \end{remark}
We now turn our attention towards constructing examples of different classes of operator tuples. The definitions of classes of operator tuples under consideration are given in Section 2. We use Proposition \ref{m5} to construct following examples. 
\begin{example} \ 
\begin{enumerate}
\item The functions $\varphi_1 (x,y)=\sqrt{x+y+1}, \varphi_1 (x,y)=\log (x+y+2),$\\$\varphi_1(x,y)=\frac{x+y+\lambda}{x+y+1} (0< \lambda <1) $ are completely alternating. Therefore each weighted translation semigroup ${S_1}_{t_1}$ with symbol $\varphi_1$ is completely hyperexpansive. Hence each pair $({S_1}_{t_1},{S_1}_{t_1})$ is a toral complete hyperexpansion and the pair $({S_1}_{t_1}/\sqrt 2,{S_1}_{t_1}/\sqrt 2)$ is a spherical complete hyperexpansion. 

\item The pair $({S_1}_{t_1},{S_2}_{t_2}),$ where $\varphi_1(x,y)=y+1,\varphi_2(x,y)=x+1$ is a toral isometry. Then the pair $({S_1}_{t_1}/\sqrt 2,{S_2}_{t_2}/\sqrt 2)$ is a spherical isometry.

\item The pair $({S_1}_{t_1},{S_2}_{t_2}),$ where $\varphi_1(x,y)=c,\varphi_2(x,y)=x+1$ is a toral 2-isometry. Then the pair $({S_1}_{t_1}/\sqrt 2,{S_2}_{t_2}/\sqrt 2)$ is a spherical 2-isometry.

\item The weighted translation semigroup corresponding to the symbol \\ $\varphi_1 (x,y)=x+y+1$ is 2-isometry. Hence the pair $({S_1}_{t_1},{S_1}_{t_1})$ is a toral 2-isometry and the pair $({S_1}_{t_1}/\sqrt 2,{S_1}_{t_1}/\sqrt 2)$ is a spherical 2-isometry. \eop
\end{enumerate}
\end{example} 

We here quote a proposition useful in constructing hyponormal tuples.
\begin{proposition} \cite[Remark 7]{At-1} If $H_1,\cdots ,H_m\in {\cal B}(H),~~H_i^*H_j=H_j^*H_i$ for $i\neq j$ and each $H_i$ is hyponormal, then the tuple $(H_1,\cdots ,H_m)$ is hyponormal. \end{proposition}
Note that the above proposition is not useful to construct hyponormal tuples of the type $S_{\bf t}$ defined in Section 2. 

Consider the pair $({S_1}_{t_1},{S_2}_{t_2}).$ It is easy to see that $${S_1}_{t_1}^*{S_2}_{t_2}={S_2}_{t_2}{S_1}_{t_1}^*$$ if and only if $$\frac{\varphi_1(x+t_1,y)\varphi_2(x+t_1,y)}{\varphi_1(x,y)\varphi_2(x+t_1,y-t_2)}=\frac{\varphi_2(x,y)\varphi_1(x+t_1,y-t_2)}{\varphi_2(x,y-t_2)\varphi_1(x,y-t_2)}$$
and $${S_2}_{t_2}^*{S_1}_{t_1}={S_1}_{t_1}{S_2}_{t_2}^*$$ if and only if $$\frac{\varphi_2(x,y+t_2)\varphi_1(x,y+t_2)}{\varphi_2(x,y)\varphi_1(x-t_1,y+t_2)}=\frac{\varphi_1(x,y)\varphi_2(x-t_1,y+t_2)}{\varphi_1(x-t_1,y)\varphi_2(x-t_1,y)}.$$
\begin{example} We now construct examples of hyponormal tuples.
\begin{enumerate}
\item
Let $b,c$ be positive real numbers. The pair $({S_1}_{t_1},{S_2}_{t_2})$ is hyponormal for the following pair of functions:
\begin{enumerate}
\item $\varphi_1(x,y)=c,~~\varphi_2(x,y)=b$ 
\item $\varphi_1(x,y)=e^x,~~\varphi_2(x,y)=e^y$
\item $\varphi_1(x,y)=c,~~\varphi_2(x,y)=e^{x+y}$
\item $\varphi_1(x,y)=e^{-x-y},~~\varphi_2(x,y)=e^{x+y}$
\item $\varphi_1(x,y)=e^{-y},~~\varphi_2(x,y)=e^{-x}$
\end{enumerate}  
\item The functions $\varphi_1(x,y)=e^{-(x+y)},\varphi_1(x,y)=\frac{1}{x+y+1}$ are completely monotone. Therefore each weighted translation semigroup ${S_1}_{t_1}$ with symbol $\varphi_1$ is subnormal. Hence by Proposition \ref{m7}, each tuple $(I,{S_1}_{t_1},{S_1}_{t_1}^2,\cdots ,{S_1}_{t_1}^{p-1}),$ \\ $p\geq 1$ is hyponormal. \eop \end{enumerate} \end{example}

Using similar techniques as discussed in Section 3, the toral analytic model for the special tuple can be constructed. Further, the properties of Taylor spectrum are similar.

\end{document}